\documentclass{amsart}%
\usepackage{amscd,amsfonts,amsmath,amsthm,amssymb,amsxtra,}
\usepackage{mathabx,mathrsfs,subfigure}
\usepackage{mathtools}
\usepackage{enumitem}
\usepackage{scalerel}
\usepackage[utf8]{inputenc}
\usepackage{newunicodechar}
\usepackage{newclude}%
\usepackage[mathscr]{eucal}
\usepackage{indentfirst}%
\usepackage{accents}%
\usepackage{latexsym}%
\usepackage{cancel}%
\usepackage{bbold}%
\usepackage{xcolor}%
\usepackage{verbatim}%
\usepackage{graphicx}
\usepackage{enumitem}%
\usepackage{setspace}%

\usepackage{rotating}%

\usepackage{pifont}%

\usepackage[english]{babel}

\usepackage{mathtools}%
\usepackage{thmtools}%
\usepackage{mathrsfs}%
\usepackage{dsfont}%

\usepackage{calc} 
\usepackage{latexsym}
\usepackage{pdfsync}
\usepackage{mdwlist}

\usepackage{tikz}%
\usepackage{tikz-cd}%
\tikzset{
  no line/.style={draw=none,
    commutative diagrams/every label/.append style={/tikz/auto=false}},
  from/.style args={#1 to #2}{to path={(#1)--(#2)\tikztonodes}}}%

\tikzset{
    rotate90/.style={anchor=south, rotate=90, inner sep=.5mm}
}%

\usepackage[all]{xy}

\usepackage[colorlinks, bookmarks=true]{hyperref}%

\numberwithin{equation}{subsection}%
\newtheorem{theorem}{Theorem}[section]
\theoremstyle{plain}

\newtheorem{corollary-definition}[theorem]{Corollary-Definition}

\newtheorem{definition}[theorem]{Definition}

\newtheorem{lemma}[theorem]{Lemma}

\newtheorem{proposition}[theorem]{Proposition}

\numberwithin{equation}{section}

\theoremstyle{definition}
\newtheorem{example}[theorem]{Example}
\newtheorem{remark}[theorem]{Remark}

\newcommand{\blank}{\hspace{0.04cm} \rule{2.4mm}{.4pt} \hspace{0.04cm} }
\newcommand{\blankd}{\hspace{0.04cm} \rule{1.5mm}{.4pt} \hspace{0.04cm} }

\DeclareMathOperator{\Ker}{\mathrm{Ker}}
\DeclareMathOperator{\Cok}{\mathrm{Coker}}

\DeclareMathOperator{\Hom}{\mathrm{Hom}}

\mathchardef\mhyphen="2D

\newcommand{\injt}{\mathfrak{s}}
\newcommand{\injc}{\mathfrak{q}}

\newcommand{\cat}[1] {\ensuremath{\mathbf{#1}}}

\newcommand{\ModL}{\mathrm{Mod}\mhyphen\Lambda}

\newcommand{\LMod}{\Lambda\mhyphen\mathrm{Mod}}

\newcommand{\ab}{\mathbf{Ab}}

\newcommand{\lra}{\longrightarrow}

\newcommand{\fp}{{\mathrm{fp}}(\LMod, \ab)}

\DeclareMathOperator{\yu}{\yen^{\bullet}}

\newcounter{hours}
\newcounter{minutes}

\newcounter{EquationCounter}[subsection]%

\begin{document}
\title[The defect, the Malgrange functor, and linear control systems]{The defect, the Malgrange functor, and linear control systems}
\author{Alex Martsinkovsky}
\address{Mathematics Department\\
Northeastern University\\
Boston, MA 02115, USA}
\email{a.martsinkovsky@northeastern.edu}

\date{\today, \setcounter{hours}{\time/60} \setcounter{minutes}{\time-\value
{hours}*60} \thehours\,h\ \theminutes\,min}

\thanks{Supported in part by the Shota Rustaveli National Science Foundation of Georgia Grant NFR-18-10849}

\dedicatory{In memory of Jeremy Russell}

\begin{abstract}

The notion of defect of a finitely presented functor on a module category is extended to arbitrary additive functors. The new defect and the contravariant Yoneda embedding form a right adjoint pair. The main result identifies the defect of the covariant Hom modulo projectives with the Bass torsion of the fixed argument. 
When applied to a linear control systems, it shows that the defect of the Malgrange functor of the system modulo projectives is isomorphic to the autonomy of the system.
Furthermore, the defect of the contravariant Hom modulo injectives is shown to be isomorphic to the cotorsion coradical of the fixed argument. Since the Auslander-Gruson-Jensen transform of cotorsion is isomorphic to torsion, the above results raise two important questions: a) what is a control-theoretic interpretation of the covariant Yoneda embedding of the Malgrange module modulo injectives, and b) what is a control-theoretic interpretation of the Auslander-Gruson-Jensen duality?
\end{abstract}

\maketitle
\tableofcontents

\section{Introduction}

\textbf{Blanket assumption}. All functors in this paper are assumed additive.
\smallskip

A general notion of torsion was introduced and studied in~\cite{MR-1, MR-2, MR-3}. Two aspects of that work deserve special mention. First, it is the utmost generality of that concept -- the definition works for arbitrary modules over arbitrary rings. For arbitrary modules over commutative domains it coincides with the classical torsion. For finitely presented modules over arbitrary rings it coincides with the Bass torsion, defined as the kernel of the canonical bidualization map. Secondly, the functorial techniques used in the definition of torsion allowed to introduce for the first time a new concept called cotorsion. It is worth mentioning that 116 years had passed from Poincar\'e's discovery of torsion (\cite[p.~33]{Dieu89}) to the introduction of cotorsion in 2016. The Auslander-Gruson-Jensen functor effects a duality between the two concepts~\cite{MR-2}.

In June 2017, in a private communication to the author, Mohamed Barakat pointed out  a connection between torsion and linear control systems (LCSs). To wit, thinking of an LCS as described by its Malgrange module (see below), one observes that its torsion submodule represents the autonomy of the system, i.e., the part of the system that cannot be controlled. It was this observation, already known to control theorists, that made this author a strong advocate of a functorial approach to the study of LCSs. 

Functor categories have several advantages over module categories. First, they have nice homological properties. Secondly, the category of modules over a ring can be embedded in a functor category via the so-called Yoneda embedding(s). Thirdly, the module category can be recovered from a functor category as a certain Serre quotient. These fundamental properties were established by M.~Auslander in his seminal paper~\cite{Aus66}. In that work he studied coherent, a.k.a. finitely presented, functors. Knowing a finite presentation of an additive functor allows for direct calculations because, in many cases, such calculations are reduced to repeated applications of the Yoneda lemma coupled with the standard exactness properties of the Hom functor. Roughly speaking, finite presentation allows to translate questions about functors back into the language of modules. This tool is completely missing when the functor is not finitely presented.

On the other hand, many naturally arising functors are not, in general, finitely presented. Examples include tensor products and their derived functors, various torsion functors and their derived functors (e.g., local cohomology), the covariant Hom functors modulo projectives, the contravariant Hom functors modulo injectives, and others. To remedy this deficiency, in this paper we introduce another ``bridge'' from  functors to modules, called the defect. For finitely presented functors this tool was already present in~\cite{Aus66}, but we were able to extend this notion to arbitrary additive functors. For a given functor, its defect is defined as the class (actually, a set) of all natural transformations from the functor to the forgetful functor. We also prove that the defect and the contravariant Yoneda embedding form a right adjoint pair.

For an additive contravariant functor, we define its defect dually as the set of all natural transformations from the forgetful functor to the functor in question. Of course, by the Yoneda lemma, this set is just the value of the functor on the regular module, and the definition remains the same as for finitely presented functors.

Significant progress in a functorial approach to the study of control systems was recently achieved by Sebastian Posur in~\cite{Pos23}, where he obtained a categorical framework for behaviors in algebraic system theory. 

The present paper  provides a functorial description of the autonomy, understood here as the  torsion submodule, showing that the autonomy is isomorphic to the defect of the (projectively) stabilized Malgrange functor, where the Malgrange functor is defined as the covariant functor represented by the Malgrange module.

Here is a brief outline of the paper. 

Section~\ref{S:Malg} deals with LCSs. We recall basic definitions and the Malgrange isomorphism.
\smallskip

In Section~\ref{S:fpfunctors} we recall basic facts about finitely presented functors. 
\smallskip

In Section~\ref{S:defect} the defect of an arbitrary additive functor is introduced. Proposition~\ref{L:op-Yoneda} shows that the Yoneda embedding and the defect form a right adjoint pair. Theorem~\ref{T:defhommodproj} computes the defect of a covariant Hom modulo projectives, which turns out, somewhat surprisingly, to be isomorphic to the Bass torsion $\mathfrak{t}$ of the fixed argument. The surprise is that, in this role, one would expect the better behaved torsion radical $\injt$, introduced in~\cite{MR-2}. However, in the case of current interest -- when the argument is finitely presented -- $\injt$ and $\mathfrak{t}$ coincide (see~\cite{MR-2} for more details).

Parallel to that, the defect of the contravariant Hom modulo injectives is isomorphic to the cotorsion $\injc$ of the fixed argument. It was shown in~\cite{MR-2} that, under the duality effected by the Auslander-Gruson-Jensen transform, $\injc$ is sent to~$\injt$. 

We recall from~\cite{MR-2} that $\injt$ is the largest subfunctor of~$\mathfrak{t}$ preserving filtered colimits. As we just remarked, it coincides with 
$\mathfrak{t}$ on finitely presented modules, but the two may differ on infinite modules. That~$\mathfrak{t}$ is a radical (in the sense of Stenstr\"{o}m) follows at once since it is the reject of the regular module or, equivalently, of the class of projectives. The torsion functor $\injt$ is a radical, too, but whether it is a reject of a suitable class is an open question. Moreover, the only known to the author proof 
that $\injt$ is a radical~\cite[Theorem~2.19]{MR-2} is based on the fact that the Bass torsion has this property. It appears that $\injt$ and $\mathfrak{t}$ are in a truly symbiotic relationship.
\smallskip

Section~\ref{S:open} contains some open questions on further possible connections with LCSs.
\smallskip

For the module-theoretic terminology used here, consult~\cite{AF92}.
\smallskip

\textbf{Acknowledgments}. I am grateful to Mohamed Barakat, Vakhtang Lomadze, Alban Quadrat, and Sebastian Posur for stimulating discussions and their patience while answering my numerous questions.

\section{The Malgrange functor of a linear control system}\label{S:Malg}

By an LCS we understand an underdetermined system of linear differential equations. Such a system can be written in the form
\begin{equation}\label{Eq:DE}
x'(t) = Ax(t) + Bu(t).
\end{equation}

Here the functions $x(t)$ are called the state variables, whereas the $u(t)$ are the so called input variables. The inputs $u(t)$ are free in the sense that they can be chosen arbitrarily. The symbols $A$ and $B$ stand for matrices with entries in a suitable ring of differential operators. We assume there are only finitely many equations and finitely many variables, making $A$ and $B$ finite.  The coefficients can be constant, polynomial, or analytic functions. To control the system~\eqref{Eq:DE} one needs output variables $y(t)$:
\begin{equation}
 y(t) = Cx(t) + Du(t)
\end{equation}
for some chosen matrices $C$ and $D$. The classical form of a control system is then given by 
\begin{equation}\label{D:system}
\begin{cases}
	x'(t) &= Ax(t) + Bu(t) 
\\
	y(t) &= Cx(t) + Eu(t)
 \end{cases}
\end{equation}
Strictly speaking, the second equation is not an intrinsic part of the system. Nevertheless, it is a common practice to refer to~\eqref{D:system} as ``the system''. Notice also that there is nothing to ``solve for" in the second equation -- this is just an explicit definition of the output.

Rewriting~\eqref{D:system} in the form $AX = 0$ (with $A$ here different from the $A$ in~\eqref{D:system}), we have that an LCS is just a system of linear equations in modules over a suitable ring $D$ of differential operators. It was a simple but important observation by Malgrange~\cite[3, $2^{\circ}$, p. 84]{Mal64}, sometimes referred to as the Malgrange isomorphism, that solutions of $AX = 0$ can be realized as homomorphisms from the module presented by the transpose $A^{T}$ of $A$. We want to rephrase this in a functorial language. 

Associated with the matrix $A$ above is the system $AX=0$ of linear equations in left $D$-modules. Given a left $D$-module $V$, the solution-set $S(V)$ of the system in $V$ is a subset of $V^{n}$, where $n$ is the number of variables in the system. This subset, as is easily seen, is also a subgroup of the underlying abelian group of 
$V^{n}$. (Since~$D$ need not be commutative, $S(V)$ is not, in general, a submodule.) As a result, we obtain an additive covariant functor $S$, the solution-set functor, from left $D$-modules to abelian groups. Let~$M$ be the left module with free presentation given by $A^{T}$. The Malgrange isomorphism is now the statement that the functor~$S$ is represented by~$M$:
\[
S \simeq (M, \blank).
\]
\begin{remark}
 In its simplest form, when $A = a \in D$  is a  1 x 1 matrix, and $D$ is a commutative ring, the Malgrange isomorphism is probably known to everybody:
 \[
 \big(D/(a), N\big) \simeq \{x \in N \mid ax = 0\},
  \]
  where $N$ is any $D$-module.
\end{remark}

We shall refer to the finitely presented left $D$-module with presentation 
matrix~$A^{T}$ as the Malgrange module of the system $AX=0$. The corresponding representable functor $(M, \blank)$ will be called the Malgrange functor of the system.

Recall that a control system is said to be controllable if any final stage can be reached, by choosing the inputs, from any initial state in finite time. Not every system is controllable, and the part of the system that cannot be controlled is known as the autonomy of the system. In terms of the Malgrange module $M$ of the system, the autonomy corresponds to the torsion submodule of $M$. This was originally observed by J.-F.~Pommaret in 1989 in the context of differential algebra. In 1991, M.~Fliess gave a module-theoretic formulation of this result. In this paper, we are interested in a functorial approach to the study of control systems, which brings about a natural question: can the autonomy be described in terms of the Malgrange functor associated with the system? The main result of this paper gives a positive answer to the algebraic reformulation of this question.  

\section{The defect of a finitely presented functor}\label{S:fpfunctors}

In this section we recall the definition and basic properties of finitely presented functors and of the category they form. In particular, we recall the definition of the defect of a finitely presented functor. The canonical reference for this material is~\cite{Aus66}.

\begin{definition}
 A functor $F$ is said to be finitely generated if there is a module $X$ and a natural transformation $\alpha : (X, \blank) \to F$ which is epic on each component.\footnote{The symbol $(X, \blank)$ denotes $\Hom(X, \blank)$.}
\end{definition}

\begin{proposition}\label{P:set}
 Let $F$ and $G$ be functors. If $F$ is finitely generated, then the class of all natural transformations $g : F \to G$, denoted by $\mathrm{Nat}(F, G)$, is a set. 
\end{proposition}
\begin{proof}
 Let $\alpha : (X, \blank) \to F$ be a componentwise epimorphism. Then, the assignment $g \mapsto g\alpha$ yields an injection $\mathrm{Nat}(F, G) \to\mathrm{Nat}((X,\blank), F)$. By the Yoneda lemma, the latter is isomorphic to $F(X)$, which is a set. Hence $\mathrm{Nat}(F, G)$ is a set, too.
\end{proof}

\begin{definition}
 A functor $F$ is said to be finitely presented if there are modules~$X$ and~$Y$ and a homomorphism $f : Y \to X$ such that the sequence
\begin{equation}\label{Eq:fp}
 (X,\blank) \overset{(f, \blankd)}\lra (Y, \blank) \lra F \lra 0
\end{equation}
 is componentwise exact.
\end{definition}

Thus every finitely presented functor is finitely generated and, as a consequence of 
Proposition~\ref{P:set}, the totality of all finitely presented functors and natural transformations between them, denoted here by~$\fp$, is a category.

The notions of monomorphism, epimorphism, and isomorphism between functors are defined componentwise and coincide with the corresponding categorical definitions.

\begin{proposition}
$\fp$ has kernels and cokernels.
\end{proposition}

\begin{proof}
See~\cite[Proposition~2.1]{Aus66}. For a more direct proof, see~\cite[Proposition~3.1]{Aus82}.
\end{proof}

\begin{proposition}
 $\fp$ is abelian.
\end{proposition}

\begin{proof}
 The canonical morphism from the coimage to the image is an isomorphism because it is an isomorphism componentwise.
\end{proof}

\begin{proposition}
 $\fp$ has enough projectives.
\end{proposition}

\begin{proof}
The projectives are precisely the representables, which can be seen by applying the Yoneda lemma.
\end{proof}

\begin{remark}
 The projectives in $\fp$ can also be characterized as the left-exact finitely presented functors. Indeed, the representables are of this form. On the other hand, a finitely presented functor clearly preserves arbitrary products. If it is left-exact, it also preserves finite limits. It follows that it preserves arbitrary limits, and by a theorem of Watts, it must be representable.
 \end{remark}                                                                                                                                                                                                                                                                                                                                                                                                                                                                                                                                                                                                                                                                                                                                                                                                                                                                                                                                                                         
 
 An important invariant of a finitely presented functor is its defect, a notion introduced  (without a name) by Auslander in~\cite[pp.~202 and 209]{Aus66}. We recall it now.
 
\begin{definition}
 Let $F$ be a covariant functor with presentation~\eqref{Eq:fp}. The module
 \[
 w(F) := \Ker f
 \]
 is called the defect of $F$.
\end{definition}

The following known result follows directly from this definition.

\begin{lemma}
The defect of a finitely presented functor is zero if and only if the functor vanishes on injectives. \qed
\end{lemma}

The notion of defect for a contravariant functor is defined similarly~[ibid.] 
\begin{definition}
Let~$F$ be a contravariant functor with presentation 
\begin{equation}\label{Eq:fp-contra}
(\blank, Y) \overset{(\blankd, f)}\lra (\blank, X) \lra F \lra 0.
\end{equation}
The module $v(F) := \Cok f$ is called the defect of $F$.
\end{definition}

It is immediate from the definition that $v(F) \simeq F(\Lambda)$. The following lemma is now obvious.

\begin{lemma}
 The defect of a finitely presented contravariant functor is zero if and only if the functor vanishes on the regular module. \qed
\end{lemma}

\begin{remark}
 The first phenomenological study of the defect of a finitely presented functor was undertaken by Jeremy Russell in~\cite{Russ16}.
\end{remark}

\section{The defect of an arbitrary additive functor}\label{S:defect}

The goal of this section is to make the notion of defect applicable to arbitrary additive functors on $\LMod$, the category of left $\Lambda$-modules, with values in abelian groups. We remind the reader that Auslander's definition applies to finitely presented functors only.

Recall that the contravariant Yoneda embedding, denoted here by $\yu$, is a functor 
$\LMod \lra \fp$ which sends a module $A$ to the representable functor $(A, \blank)$. 
In particular, $\yu(\Lambda)$ is the forgetful functor $(\Lambda, \blank)$. We can now introduce a defect of an arbitrary additive functor

\begin{definition}
 Let $F : \LMod \lra \ab$ be an arbitrary additive covariant functor. Define the defect $w(F)$ of $F$ by setting 
 \[
 w(F) := (F, \yu(\Lambda)),
 \]
 where the right-hand side is the class of all natural transformations $F \to \yu(\Lambda)$.
 \end{definition}
 
 By Proposition~\ref{P:set}, if $F$ is finitely generated, then $w(F)$ is a set. However, 
since the forgetful functor is a biadjoint, the following theorem, due to G.~M.~Kelly~\cite[Theorem (1.5)]{Lin65}, shows that the finiteness assumption can be dropped.\footnote{I am grateful to Alex Sorokin for bringing this result to my attention.}

\begin{theorem}\label{T:Kelly}
 Let $F, G : \cat{L} \lra \cat{K}$ be two functors. The class $(F,G)$ of natural transformations from $F$ to $G$ constitutes a set provided either
\begin{enumerate}[label=\normalfont(\roman*)]
 \item  $\cat{K}$ has a generator and $F$ has a left adjoint, or
 \item  $\cat{K}$ has a cogenerator and $G$ has a right adjoint, or
 \item  $\cat{L}$ has a generator and $F$ has a right adjoint, or
 \item  $\cat{L}$ has a generator and $G$ has a left adjoint.
 \qed
\end{enumerate}

\end{theorem}
Thus $w(F)$ is always a set. Notice also that the symbol $w$ is, in fact, an additive contravariant functor from the large category of additive functors $\ModL \to \ab$ to abelian groups. Moreover, the bimodule structure on $\Lambda$ gives rise to the structure of a left $\Lambda$-module on each component of the defect. In summary,
we have

\begin{lemma}\label{L:def-rep}
 The defect is an additive contravariant functor from the large functor category back to modules:
\[
w : (\LMod, \ab) \lra \LMod.
\]
It is represented by the forgetful functor $\yu(\Lambda)$ and, in particular, is 
left-exact.\footnote{It can be shown that the defect, viewed as a functor on finitely presented functors, is a biadjoint(\cite{RD18}). In particular, it is exact. However, in general, $w$ is not right-exact.} \qed
\end{lemma}

In preparation for the next proposition we sketch some basic properties of additive functors. Let ${\{F_{i} : \LMod \lra \ab\}}$ be a family of such functors. Define an additive functor $\prod F_{i}$ componentwise: $\big(\prod F_{i}\big) (X) := \prod \big(F_{i} (X)\big)$. Notice that the right-hand side is unambiguously defined as a product of abelian groups. Furthermore, the structure maps $ \pi_{j}(X) : \prod \big(F_{i} (X)\big) \lra F_{j} (X)$ give rise to natural transformations $ \pi_{j} : \prod F_{i} \lra F_{j}$. It is now a straightforward verification that $\prod F_{i}$ is a product of the $F_{i}$ with structure maps $\pi_{j}$. As an immediate consequence, we have

\begin{lemma}
 For any functor $F$ and any family ${\{F_{i}\}}$ of functors there is a canonical isomorphism $(F, \prod F_{i}) \cong \prod (F, F_{i})$. \qed
\end{lemma}

The proof of the next lemma is also straightforward. 
\begin{lemma}
 The contravariant Yoneda embedding sends direct sums to direct products. In particular, $\yu(\coprod A_{i}) \simeq \prod \big(\yu(A_{i})\big)$, where the $A_{i}$ are $\Lambda$-modules.\footnote{This is a particular case of a more general observation: products in the large functor category can be defined and computed componentwise. While the proof of the lemma is very simple, the reader should not fall into a trap and claim that the lemma is completely trivial. Indeed, the product in the right-hand side is to be taken in the functor category!} \qed
\end{lemma}

We are now ready to prove  

\begin{proposition}\label{L:op-Yoneda}
 The defect and the contravariant Yoneda embedding form a right adjoint pair, i.e., for any additive covariant functor $F$ and any $\Lambda$-module~$A$, there is an isomorphism 
\begin{equation}\label{D:r-adj}
 (F, \yen^{\bullet}(A)) \simeq \big(A, w(F)\big).
\end{equation}
 natural in $F$ and $A$. 
\end{proposition}

\begin{proof}
 We are going to look at the left-hand side of the claimed isomorphism as the value of the additive contravariant functor $G:= \big(F, \yen^{\bullet}(\blank)\big)$ on $A$. This is a functor from $\LMod$ to (large) abelian groups.\footnote{Notice that, by Theorem~\ref{T:Kelly}, $G(\Lambda)$ is a small abelian group.} It is well-known (see, for example, the discussion preceding~\cite[Lemma~3.16]{MR-1}) that the canonical natural transformation from any additive contravariant functor to the result of the Yoneda embedding of its value on $\Lambda$ is an isomorphism on $\Lambda$. Thus, we have a canonical natural transformation $\sigma : G \to \big(\blank, G(\Lambda)\big)$, which evaluates to an isomorphism on~$\Lambda$. The discussion preceding this proposition shows that~$G$ converts direct sums to direct products, and therefore 
 $\sigma$ is an isomorphism on all projectives. Since the codomain of $\sigma$ is clearly left-exact, $\big(\blank, G(\Lambda)\big)$ is the zeroth right-derived functor $R_{0}G$ of~$G$.\footnote{The use of a subscript rather than a superscript is only for mnemonic purposes to indicate that~$G$ is a contravariant functor.} But, as is easily checked, $G$ is itself left-exact and therefore $\sigma$ is an isomorphism. The proof of the functoriality in each variable is straightforward.
\end{proof}

\begin{remark}
 The above proposition gives yet another proof that $w$ is left-exact. 
\end{remark}

\begin{remark}
 J.~Fisher Palmquist and D. Newell introduced a notion of the dual of an additive covariant functor~\cite[pp. 299-300]{FishNew71}. In our notation, the value $F^{\ast}(A)$ of the dual of~$F$ on $A$ is defined as $(F, \yen^{\bullet}(A))$.\footnote{The reader should notice that the $(\blank)^{\scaleto{\ast\mathstrut}{5pt}}$ operation flips the variance of the functor.}\textsuperscript{,}\footnote{Similarly, if $F$ is contravariant, $F^{\scaleto{\ast\mathstrut}{5pt}}(A)$ is defined as 
 $(F, \yen_{\scaleto{\bullet\mathstrut}{5pt}}(A))$, where 
 $\yen_{\scaleto{\bullet\mathstrut}{5pt}}$ denotes the covariant Yoneda embedding.} Thus the lemma says there is a functorial in $F$ isomorphism $F^{\ast} \simeq (\blank, w(F))$. In particular,  $w(F)$ represents~$F^{\ast}$.
\end{remark}

\begin{proposition}
For finitely presented functors the new defect is isomorphic to the defect defined by Auslander.
\end{proposition}

\begin{proof}
Map the presentation~\eqref{Eq:fp} on page~\pageref{Eq:fp} into $(\Lambda, \blank)$ and use the Yoneda lemma together with the left-exactness of the Hom functor.
\end{proof}

We can now show that the just defined defect of the covariant Hom modulo projectives is isomorphic to the Bass torsion of the representing object. By the  Bass torsion $\mathfrak{t}(A)$ of a $\Lambda$-module $A$ we understand the kernel of the canonical evaluation map $A \to A^{\ast\ast}$. It consists of all elements of $A$ on which all linear forms $A \to \Lambda$ vanish.

\begin{theorem}\label{T:defhommodproj}
 Let $\Lambda$ be an arbitrary ring. For any left $\Lambda$-module $A$, there is  a  canonical isomorphism 
 \[
 \varepsilon : w(\underline{A, \blank}) \overset{\simeq}{\lra} \mathfrak{t}(A),
 \]
 functorial in $A$.
\end{theorem}
 
\begin{proof}
 Start with the defining exact sequence 
\begin{equation}\label{L:st}
0 \lra P(A, \blank) \overset{\iota}{\lra} (A, \blank) \overset{\pi} \lra (\underline{A, \blank}) \lra 0,
 \end{equation}
 where $P(A, \blank)$ denotes the subfunctor of $(A, \blank)$ consisting of all maps factoring through projectives, and pick any $\alpha \in \big((\underline{A, \blank}), \yu(\Lambda)\big)$. By the Yoneda lemma, $\alpha \pi = (f_{\alpha}, \blank)$ for some uniquely determined $f_{\alpha} : \Lambda \to A$. Since $\Lambda$ is projective, 
 $\pi_{\Lambda} = 0$  and therefore $(f_{\alpha}, \Lambda) = 0$. This means that 
 $h f_{\alpha} = 0$ for any linear form $h : A \to \Lambda$. In particular, 
 $h f_{\alpha}(1) = 0$, i.e., any linear form vanishes on $f_{\alpha}(1)$, and therefore $f_{\alpha}(1) \in \mathfrak{t}(A)$. Thus, we have a well-defined canonical 
 $\Lambda$-linear map 
\[
\varepsilon : \big((\underline{A, \blank}), \yu(\Lambda)\big) \lra \mathfrak{t}(A) : \alpha \mapsto f_{\alpha}(1).
\]
If $f_{\alpha}(1) = 0$ then $f_{\alpha} = 0$ and therefore $\alpha\pi = 0$. Since $\pi$  is epic, this shows that $\alpha =0$ and therefore $\varepsilon$ is injective.

To show that $\varepsilon$ is surjective, pick an arbitrary $b \in \mathfrak{t}(A)$ and define a homomorphism $f_{b} : \Lambda \to A $ of left $\Lambda$-modules by 
 sending $ \lambda \mapsto \lambda b$. This map gives rise to a natural transformation $\beta : = (f_{b}, \blank) : (A, \blank) \to (\Lambda, \blank)$. Since any linear form $A \to \Lambda$ vanishes on $b$, we have that $\beta \iota = 0$ 
and~$\beta$ must be of the form $\alpha \pi$ for some $\alpha \in \big((\underline{A, \blank}), \yu(\Lambda)\big)$. The foregoing arguments now show that $\varepsilon(\alpha) = b$, proving that $\varepsilon$ is surjective.

Verification of the functoriality of $\varepsilon$ is straightforward.
\end{proof}

\begin{remark}
 The defect of the covariant Hom modulo injectives is 0. Indeed, a presentation of  
$(\overline{A, \blank})$ is determined by a short exact sequence $0 \to A \overset{f}\to I \to \Sigma A \to 0$, and therefore $w(\overline{A, \blank}) = \Ker f =0$.
\end{remark}

For the sake of completeness, we briefly discuss the contravariant case. Recall that the defect of the finitely presented contravariant functor $F$ with presentation~\eqref{Eq:fp-contra} was defined as $v(F) := \Cok f = F(\Lambda)$. This makes extending the covariant defect to arbitrary additive functors obvious.

\begin{definition}
 Let $F$ be an additive contravariant functor. The defect $v(F)$ is defined by setting $v(F) := F(\Lambda)$.
\end{definition}

It is immediate that $v$ is an exact covariant functor from contravariant additive functors back to modules. Straight from the definition we have $v(F)$ is zero if and only if $F$ vanishes on the regular module. 
\begin{example}
Trivially, $v(\underline{\blank, A}) \simeq 0$.
\end{example}

\begin{remark}
The reader may wonder why the definition of $v$ looks so different from that of $w$. In fact, the two definitions are similar: by the Yoneda lemma, 
$F(\Lambda) \simeq \big(\yen_{\bullet}(\Lambda), F\big)$, where $\yen_{\bullet}$ stands for the covariant Yoneda embedding. Thus the definition of $v$ can also be stated as
\[
v(F) = \big(\yen_{\bullet}(\Lambda), F\big).
\]
\end{remark}

\begin{example}
 A natural companion for the covariant Hom modulo projectives is the contravariant Hom modulo injectives, $(\overline{\blank, A})$. In general, this functor is not finitely presented. Computing its defect, we have
\[
v(\overline{\blank, A}) \simeq  (\overline{\Lambda, A}) = \injc(A),
\]
where $\injc$ is the cotorsion coradical introduced in~\cite{MR-2}. Thus the defect of the contravariant Hom modulo injectives is isomorphic to the cotorsion of the fixed argument. 
\end{example}

Interpreting the foregoing results in the context of linear control systems (see Section~\ref{S:Malg}), we have
\begin{theorem}
 The autonomy of a linear control system is isomorphic to the defect of its Malgrange functor modulo projectives. \qed
\end{theorem}

\section{Concluding remarks and open questions}\label{S:open}

In the previous section we saw that the defect of $(\underline{A, \blank})$ is isomorphic the Bass torsion $\mathfrak{t}(A)$ of $A$. If $A$ is finitely presented, then 
$\mathfrak{t}(A)$ coincides with $\injt(A)$ (\cite[Proposition 2.13]{MR-2}), where 
$\injt$ is the torsion radical introduced in [ibid.]. Thus, assuming that a control system is represented by finitely many equations in finitely many unknowns, both $\injt$ and 
$\mathfrak{t}$ represent the autonomy. This brings about 

\textbf{Question A}. Which torsion radical, if any, $\injt$ or $\mathfrak{t}$, represents the autonomy of a general system?
\medskip

It was shown in~\cite{MR-2} that $\injt$ is the largest subfunctor of $\mathfrak{t}$ preserving filtered colimits. This property prevents the pathological behavior of torsion 
responsible for the fact that the Bass torsion of the additive group of the rationals coincides with that group. This yields

\textbf{Question B}. Is it possible to view the autonomy as a subfunctor of the identity functor on a suitable category? If so, does it preserve filtered colimits?
\medskip

Recall that the Auslander-Gruson-Jensen (AGJ) transform of an additive functor~$F$ is defined by the formula $DF(N) : = (F, N \otimes \blank)$ (or $(F, \blank \otimes N)$). This is an additive functor on modules over the opposite ring. It was shown in~\cite{MR-2}, that the AGJ transform of the cotorsion $\injc = (\overline{\Lambda, \blank})$ is isomorphic to $\injt$. It is yet another advantage of $\injt$ over the Bass torsion. Returning to control systems, we have 

\textbf{Question C}. Given a control system with Malgrange module $M$, what is the control-theoretic interpretation of the cotorsion $(\overline{\Lambda, M})$ of the Malgrange module? Recall that the cotorsion of a module is simply its quotient by the trace of the class of injectives.
\medskip

For certain classes of LCSs, one has a duality between the controllability of a system and the observability of the dual system (see~\cite{Fuhr81, NW88, vdS91, Rud96, L21}).  The module-theoretic and functorial analogs of the autonomy discussed above beg for the next question.

\textbf{Question D}. Is there a control-theoretic interpretation of the Auslander-Gruson-Jensen duality?


\begin{thebibliography}{99}

\bibitem{AF92}
{F.~W. Anderson and K.~R. Fuller.
\newblock {\em Rings and categories of modules}, volume~13 of {\em Graduate Texts in Mathematics}.
\newblock Springer-Verlag, New York, second edition, 1992.}

\bibitem{Aus66}
M.~Auslander.
\newblock Coherent functors.
\newblock In {\em Proc. {C}onf. {C}ategorical {A}lgebra ({L}a {J}olla,
  {C}alif., 1965)}, pages 189--231. Springer, New York, 1966.

\bibitem{Aus82}
M.~Auslander.
\newblock A functorial approach to representation theory.
\newblock In {\em Representations of algebras ({P}uebla, 1980)}, volume 944 of
  {\em Lecture Notes in Math.}, pages 105--179. Springer, Berlin-New York,
  1982.


  

\bibitem{RD18}
S.~Dean and J.~Russell.
\newblock The {A}uslander-{G}ruson-{J}ensen recollement.
\newblock {\em J. Algebra}, 511:440--454, 2018.


\bibitem{Dieu89}
J.~Dieudonn\'{e}.
\newblock {\em A history of algebraic and differential topology. 1900--1960}.
\newblock Birkh\"{a}user Boston, Inc., Boston, MA, 1989.
  
  \bibitem{FishNew71}
J.~Fisher Palmquist and D.~C. Newell.
\newblock Bifunctors and adjoint pairs.
\newblock {\em Trans. Amer. Math. Soc.}, 155:293--303, 1971.

\bibitem{Fuhr81}
P.~A. Fuhrmann.
\newblock Duality in polynomial models with some applications to geometric
  control theory.
\newblock {\em IEEE Trans. Automat. Control}, 26(1):284--295, 1981.


%

\bibitem{Lin65}
F.~E.~J. Linton.
\newblock Autonomous categories and duality of functors.
\newblock {\em J. Algebra}, 2:315--349, 1965.

\bibitem{L21}
V.~Lomadze.
\newblock Duality for multidimensional linear systems with homological
  dimension {$\leq 1$}.
\newblock {\em SIAM J. Control Optim.}, 59(1):417--433, 2021.


\bibitem{Mal64}
B.~Malgrange.
\newblock Systemes diff{\'e}rentiels {\`a} coefficients constants.
\newblock Semin. {Bourbaki} 15 (1962/63), {No}. 246, 11 p. (1964).

\bibitem{MR-1}
A.~Martsinkovsky and J.~Russell.
\newblock Injective stabilization of additive functors, {I}. {P}reliminaries.
\newblock {\em J. Algebra}, 530:429--469, 2019.

\bibitem{MR-2}
A.~Martsinkovsky and J.~Russell.
\newblock Injective stabilization of additive functors. {II}. ({C}o)torsion and
  the {A}uslander-{G}ruson-{J}ensen functor.
\newblock {\em J. Algebra}, 548:53--95, 2020.

\bibitem{MR-3}
A.~Martsinkovsky and J.~Russell.
\newblock Injective stabilization of additive functors, {III}. {A}symptotic
  stabilization of the tensor product.
\newblock {\em Algebra Discrete Math.}, 31(1):120--151, 2021.

\bibitem{NW88}
J.~W. Nieuwenhuis and J.~C. Willems.
\newblock Duality for linear time invariant finite-dimensional systems.
\newblock In {\em Analysis and optimization of systems ({A}ntibes, 1988)},
  volume 111 of {\em Lect. Notes Control Inf. Sci.}, pages 13--21. Springer,
  Berlin, 1988.


\bibitem{Pos23}
S.~Posur.
\newblock An abelian ambient category for behaviors in algebraic systems  theory, 2023. arXiv 2303.02636
  
  \bibitem{Rud96}
J.~Rudolph.
\newblock Duality in time-varying linear systems: a module-theoretic approach.
\newblock {\em Linear Algebra Appl.}, 245:83--106, 1996.

\bibitem{Russ16}
J.~Russell.
\newblock Applications of the defect of a finitely presented functor.
\newblock {\em J. Algebra}, 465:137--169, 2016.

\bibitem{vdS91}
A.~J. van~der Schaft.
\newblock Duality for linear systems: external and state space characterization
  of the adjoint system.
\newblock In {\em Analysis of controlled dynamical systems ({L}yon, 1990)},
  volume~8 of {\em Progr. Systems Control Theory}, pages 393--403.
  Birkh\"{a}user Boston, Boston, MA, 1991.


\end{thebibliography}
\end{document}